\documentclass{article}

\usepackage{mathrsfs}
\usepackage[mathcal]{euscript}
\usepackage{amsthm}
\usepackage{amsmath}

\def\R{\bf{R}}
\def\C{\bf{C}}
\def\fc{\mathscr{C}}
\def\fcc{\mathscr{K}}

\def\supp{S}
\def\meas{\mathscr{M}}
\def\comps{\mathcal{K}}
\def\var{\mathrm{Var\,}}
\def\pair#1#2{\langle\, #1,  #2\,\rangle}
\def\bpair#1#2{\Bigl\langle\, #1,  #2\,\Bigr\rangle}

\newtheorem*{theorem*}{Theorem}
\newtheorem{lemma}{Lemma}

\theoremstyle{definition}
\newtheorem*{definition*}{Definition}

\author{Yu. A. Chapovsky}
\title{Existence of an invariant measure on a hypergroup}
\date{}
\begin{document}
\maketitle
\begin{abstract}
We prove existence of an invariant measure on a hypergroup.
\end{abstract}

\section{Introduction}

The notion of a hypergroup, more exactly, a DJS-hypergroup has
appeared around 1973 in the independent works of
C.~Dunkl~\cite{dunkl73:_struc}, R.~Jewett\cite{jewett75:_spaces}, and
R.~Spector~\cite{spector75:_aperc}.  Structures more general than
commutative DJS-hypergroups, called commutative hypercomplex systems,
were introduced by Yu.~M.~Berezansky and
S.~G.~Krein~\cite{berezansky50:_contin,berezansky57:_hyper}
(1950,~1957); for a detailed account of the theory of hypercomplex
systems,
see~\cite{berezansky_kal92:_harmon_analy_hyper_system}. Introduction
of the above structures can be considered as a development of the
theory of generalized translation operators introduced and
studied by J.~Delsart~\cite{delsart38:_sur_formul_taylor} (1938), see
also B.~Levitan~\cite{levitan73}.

At the present time, the theory of hypergroups is well developed, see,
e.g.~\cite{bloom_heyer91:_harmon_analy}. However, much of the theory
relies on existence of an invariant measure. Its existence has been
proved only in particular cases, --- for hypergroups that are either
compact or discrete~\cite{jewett75:_spaces,spector75:_aperc}, or
locally compact and commutative~\cite{spector78:_meas_inv}. In the
case of a general locally compact hypergroup, a left (or right)
invariant measure was assumed to exist, although existence of a
subinvariant measure has been proved in~\cite{spector78:_meas_inv}.
In this paper, we prove that a general (locally compact) hypergroup
possesses a left (hence, a right) invariant measure.

The paper is organized as follows. 

Section~\ref{Sec:def+not} introduces some notations used in the paper,
recalls the definition of a hypergroup and some well known facts. 

Section~\ref{Sec:main} contains a proof of the main result. The proof
is subdivided into several lemmas.

\section{Notations and definitions}\label{Sec:def+not}

All undefined topological notions agree
with~\cite{kelley75:_general_topology}. Let $Q$ be a Hausdorff locally
compact topological space.  The linear space of complex-valued
continuous functions on $Q$ is denoted by $\fc$, the subspace of $\fc$
of functions approaching zero at infinity is denoted by $\fc_0$.  The
space $\fc_0$ is endowed with the norm $\|f\|=\sup_{t\in Q}
|f(t)|$. By $\fcc$, we denote the linear subspace of $\fc_0$ of
functions with compact supports. Support of $f \in \fcc $ is denoted
by $\supp(f)$. The linear subspace of $\fcc$ of functions with the
supports in a compact set $K$ is denoted by $\fcc(K)$. The set of
nonnegative continuous functions with compact supports will be denoted
by $\fcc_+$, and $\fcc_+^*=\fcc_+\setminus\{0\}$; the subset of
$\fcc_+$ (resp., $\fcc_+^*$) of functions with the supports in a
compact set $K$ is denoted by $\fcc_+(K)$ (resp. $\fcc_+^*(K)$). For a
compact set $K$, we denote by $1_K$ a function in $\fcc_+$ such that
$1_K(s)=1$ for all $s \in K$. By $1_Q$, we denote the constant
function, $1_Q(s)=1$ for all $s \in Q$.

A measure is understood as a complex Radon
measure~\cite{bourbaki04:_integ_I} on $Q$. The linear space of complex
Radon measures, over the field $\C$ of complex numbers, is denoted by
$\meas$. For a measure $\mu$, its norm is $\|\mu\| = \sup_{f\in \fcc,\
  \|f\|\leq 1} |\mu(f)|$. The subspace of $\meas$ of bounded (resp.,
compactly supported) measures is denoted by $\meas_b$ (resp.,
$\meas_c$). The subset of $\meas$ of nonnegative (resp., probability)
measures is denoted by $\meas_+$ (resp., $\meas_p$). For a measure
$\mu \in \meas_+$, its support is denoted by $\supp(\mu)$. If $\mu \in
\meas_+\cap \meas_b$, then $\|\mu\|=\mu(1_Q)$. The Dirac measure at a
point $s\in Q$ is denoted by $\varepsilon_s$. The integral of $f\in
\fcc(F)$ with respect to a measure $\mu \in \meas$ is denoted by
$\mu(f)=\pair{f}{\mu}=\int_F \pair{f}{\varepsilon_t}\, d\mu(t)=\int_F
f(t)\, d\mu(t)$.

The system of all compact neighborhoods of a point $s\in Q$ is denoted
by~$\comps_s$.

\medskip 

A hypergroup is undersood in the sense of
R.~Spector~\cite{spector75:_aperc}, i.e., it is a locally compact
topological space $Q$ such that $\meas_b$ is endowed with the
structure of a Banach algebra. The multiplication, called composition
and denoted by~$*$, is subject to the following conditions:
\begin{itemize}
\item [$(H_1)$] $\meas_p * \meas_p \subset \meas_p$;
\item [$(H_2)$] the convolution $\meas_p \times \meas_p \ni
  (\mu_1,\mu_2)\mapsto \mu_1*\mu_2\in \meas_p$ is separably
  continuous with respect to the weak topology $\sigma(\meas_b, \fc_0)$ on
  $\meas_b$;
\item [$(H_3)$] the map $Q\times Q\ni (s,t) \mapsto
  \varepsilon_s*\varepsilon_t \in \meas_b$ is continuous with respect
  to the weak topology $\sigma(\meas_b, \fc_0)$ on $\meas_b$;
\item [$(H_4)$] there is a point $e\in Q$, which is necessarily
  unique, such that $\varepsilon_e* \mu = \mu * \varepsilon_e$ for all $\mu
  \in \meas_b$;
\item [$(H_5)$] there is an involutive homeomorphism $Q\ni t \mapsto
  \check t \in Q$ such that its continuation to $\meas_b$ satisfies
  $(\mu_1*\mu_2)\,\check{}=\check \mu_2* \check\mu_1$;
  in particular, $\check e=e$;
\item [$(H_6)$] for two points $t,s\in Q$, the condition $t=s$ is
  equivalent to the condition $e \in S(\varepsilon_t*\check \varepsilon_s)$;
\item [$(H_7)$] for every compact subset $K$ of $ Q$ and any
  neighborhood $V$ of $K$ there is a neighborhood $U$ of $e$ such that
  \begin{itemize}
  \item [$1)$] if $\supp(\mu) \subset K$ and $\supp(\nu) \subset U$,
    then $\supp(\mu*\nu)\subset V$ and $\supp(\nu*\mu)\subset V$;
  \item[$2)$] if $\supp(\mu) \subset K$ and $\supp(\nu)
    \subset Q\setminus V$, then the sets $\supp(\mu*\check \nu)$,
    $\supp(\check \nu*\mu)$, $\supp(\check \mu*\nu)$,
    $\supp(\nu*\check \mu)$ are disjoint with $U$.
  \end{itemize}
\end{itemize}

For $\mu \in \meas_b$ and $f \in \fcc$, the convolutions
$\mu*f$ and $f*\mu$ are defined by
\begin{equation}
  \label{eq:function_measure_convolution}
  (\mu*f)(t)=\pair{f}{\check \mu*\varepsilon_t},
  \qquad
  (f*\mu)(t)=\pair{f}{\varepsilon_t*\check\mu}.
\end{equation}
We have~\cite[Theorems~1.3.2, 1.3.3]{spector75:_aperc} that
\begin{equation}
  \label{eq:function_measure_convolution_properties}
  \begin{split}
    \pair{\mu*f}{\sigma}=\pair{f}{\check \mu*\sigma},
    \qquad
    \pair{f*\mu}{\sigma}=\pair{f}{\sigma*\check \mu},\\
    (\mu*\nu)*f=\mu*(\nu*f),
    \qquad
    f*(\mu*\nu)=(f*\mu)*\nu
  \end{split}
\end{equation}
for $f\in\fcc$ and $\mu,\nu,\sigma \in \meas_b$.

For subsets $A $ and $B$ of $Q$,  we use the notations
$
  A\cdot B=\overline{\bigcup_{a\in A,\, b\in B}\supp(\epsilon_a*\epsilon_b)};
$ 
we write $aB$ for $\{a\}\cdot B$.

For $\mu,\nu \in \meas_+\cap \meas_b$, we have~\cite[Theorem
I-2]{spector78:_meas_inv} that
$\supp(\mu*\nu)=\supp(\mu)\cdot\supp(\nu)$. For $f \in \fcc_+$ and
$\mu \in \meas_+\cap \meas_b$,~\cite[Theorem~I.8]{spector78:_meas_inv}
gives $\supp(\mu*f)=\supp(\mu)\cdot \supp(f)$ and
$\supp(f*\mu)=\supp(f)\cdot \supp(\mu)$.

For a measure $\mu \in \meas$ and $g\in \fc$, the measure $h\mu$ is
defined by $\pair{f}{h\mu}=\pair{fh}{\mu}$ for $f\in\fcc$; it is
clear that $h\mu \in \meas_c$ for $h \in \fcc$.

\begin{definition*}
  A positive Radon measure $\chi$ on a hypergroup $Q$ is called
  \emph{left invariant}
  \begin{equation}
    \label{eq:left-invariant-measure-def}
    \pair{\varepsilon_s*f}{\chi}=\pair{f}{\chi}
  \end{equation}
 for all $f\in \fcc$ and $s\in Q$.
\end{definition*}

\section{Main results}\label{Sec:main}

\begin{lemma}
  For $f \in \fcc$, define $\check f(s)=f(\check s)$. Then, for $\mu,
  \nu \in \meas_b$, and $f\in \fcc$, we have the following:
  \begin{equation}
    \label{eq:*-v-identities}
    \begin{array}{c}
      (\mu*f)\,\check{}=\check f*\check \mu,
      \qquad
      (f*\mu)\,\check{}=\check \mu*\check f,\\[2mm]
      \pair{\mu*f}{\nu}=\pair{\nu*\check f}{\mu}.
    \end{array}
  \end{equation}
\end{lemma}

\begin{proof}
  Since $\check f(s)=f(\check s)=\pair{f}{\check
    \varepsilon_s}$, 
  using~(\ref{eq:function_measure_convolution}) we have
  \begin{align*}
    (\mu*f)\, \check{}\,(s)&=\pair{\mu*f}{\check \varepsilon_s}=
    \pair{f}{\check \mu*\check \varepsilon_s}=
    \pair{f}{(\varepsilon_s*\mu)\,\check{}}= \pair{\check
      f}{\varepsilon_s*\mu}
    \\
    &=\pair{\check f * \check \mu}{\varepsilon_s},
  \end{align*}
  thus the first identity in~(\ref{eq:*-v-identities}) follows. The
  second one is proved similarly. The third one is obtained as
  follows:
  \begin{equation*}
    \pair{\mu*f}{\nu}=
    \pair{(\mu*f)\, \check{}\,}{\check \nu}=
    \pair{\check f*\check \mu}{\check \nu}=
    \pair{\check f}{\check \nu*\mu}=
    \pair{\nu*\check f}{\mu}.
  \end{equation*}
\end{proof}

\begin{lemma}
  Let $\mu_0 \in \meas_b$ be an arbitrary bounded positive Radon
  measure such that $\supp(\mu_0)=Q$.  For $U\in \comps_e$ and
  a function $g \in \fcc^*_+(U)$, define
  \begin{equation}
    \label{eq:mu_U_definition}
    \tilde \chi_g=\frac{1}{\mu_0*g}\mu_0.
  \end{equation} 
  Then $\tilde\chi_g$ is a Radon measure, $\supp(\tilde
  \chi_g)=Q$. Moreover, for any $f \in \fcc_+$ and $\epsilon>0$ there
  exists  $U_f\in \comps_e$ such that, for any $g\in \fcc^*_+(U_f)$,
  \begin{equation}
    \label{eq:main_identity}
    \| f- (f\tilde \chi_g)*g\|<\epsilon.
  \end{equation}
\end{lemma}

\begin{proof}
  It follows from~\cite[Theorem I.5]{spector78:_meas_inv} that
  $\mu_0*g \in \fc$ and is bounded. It is easy to see that
  $(\mu_0*g)(t)>0$ for all $t \in Q$. Thus, $\tilde \chi_g$ is a Radon
  measure.  It is evident that $\supp(\tilde \chi_g)=\supp(\mu_0)=Q$.

  Now, let us prove~(\ref{eq:main_identity}). First of all note that
  $\supp(f\tilde \chi_g)=\supp(f)$ is compact. Let us fix a compact
  neighborhood $V$ of the compact set $\supp(f)$. By axiom~($H_7$)
  there is a compact neighborhood $U_0$ of $e$ such that
  $\supp(f)\cdot  U_0 \subset V$. Hence, for any $U\subset  U_0$
  and $F \in \fcc_+^*(U)$, we have $\supp(f)\subset \supp((f\tilde
  \chi_g)*g)\subset V$.
  
  Let now $s\in Q$ and consider 
  \begin{equation}\label{eq:main_identity_s}
    \bigl|f(s)- \bigl((f\tilde \chi_g)*g\bigr)(s)\bigr|.
  \end{equation}
  If $s\notin V$, then the expression in~(\ref{eq:main_identity_s}) is
  zero. Thus, let $s \in V$. Using
  identities~(\ref{eq:function_measure_convolution})
  and~(\ref{eq:*-v-identities}) we have the following:
  \begin{align*}
    \bigl|f(s)- \bigl((f\tilde \chi_g)*g\bigr)(s)\bigr|&=
    \bigl|\pair{f-(f\tilde \chi_g)*g }{\varepsilon_s}\bigr|\\
    &=\bigl| \pair{f}{\varepsilon_s}-\pair{(f\tilde \chi_g)*g}{\varepsilon_s}
    \bigr|\\
    &=\bigl|
    \pair{f}{\varepsilon_s}\frac{\pair{\mu_0*g}{\varepsilon_s}}
    {\pair{\mu_0*g}{\varepsilon_s}} -\pair{\bigl( \frac{f}{\mu_0*g}
      \mu_0 \bigr) *g}{\varepsilon_s} \bigr|
    \\
    &=\bigl| \pair{\varepsilon_s*\check g}{\mu_0}
    \frac{\pair{f}{\varepsilon_s}} {\pair{\mu_0*g}{\varepsilon_s}} -
    \pair{\varepsilon_s*\check g}{\frac{f}{\mu_0*g}\mu_0}
    \bigr| \\
    &= \frac{1}{\pair{\mu_0*g}{\varepsilon_s}} \Bigl|
    \bpair{\varepsilon_s*\check g} {\pair{f}{\varepsilon_s} \mu_0 -
      \frac{f \pair{\mu_0*g}{\varepsilon_s}}{\mu_0*g}\mu_0} \Bigr|
    \\
    &= \frac{1}{\pair{\mu_0*g}{\varepsilon_s}} \Bigl|
    \bpair{\varepsilon_s*\check g} {\Bigl(\pair{f}{\varepsilon_s} -
      \frac{f \pair{\mu_0*g}{\varepsilon_s}}{\mu_0*g}\Bigr)\mu_0}
    \Bigr|
    \\
    &= \frac{1}{\pair{\mu_0*g}{\varepsilon_s}} \Bigl|\int_{
      \supp(\varepsilon_s*\check g)} \pair{\varepsilon_s*\check
      g}{\varepsilon_t}\
    \\
    &\qquad\times \Bigl(\pair{f}{\varepsilon_s} - \pair{f}
    {\varepsilon_t}
    \frac{\pair{\mu_0*g}{\varepsilon_s}}{\pair{\mu_0*g}{\varepsilon_t}}
    \Bigr)\, d\mu_0(t) \Bigr|
    \\
    &\leq
    \frac{1}{\pair{\mu_0*g}{\varepsilon_s}}
    \int_{\supp(\varepsilon_s*\check g)} 
    \pair{\varepsilon_s*\check g}{\varepsilon_t}\, d\mu_0(t)
    \\
    &\qquad\times
    \sup_{t \in \supp(\varepsilon_s*\check g)}
    \Bigl| \pair{f}{\varepsilon_s}
  -\pair{f}{\varepsilon_t}
  \frac{\pair{\mu_0*g}{\varepsilon_s}}{\pair{\mu_0*g}{\varepsilon_t}}
  \Bigr|
  \\
  &=
  \frac{1}{\pair{\mu_0*g}{\varepsilon_s}}
  \pair{\varepsilon_s*\check g}{\mu_0} 
  \\
  &\qquad \times
    \sup_{t \in \supp(\varepsilon_s*\check g)}
    \Bigl| \pair{f}{\varepsilon_s}
  -\pair{f}{\varepsilon_t}
  \frac{\pair{\mu_0*g}{\varepsilon_s}}{\pair{\mu_0*g}{\varepsilon_t}}
  \Bigr|
  \\
  &=
    \sup_{t \in \supp(\varepsilon_s*\check g)}
    \Bigl| \pair{f}{\varepsilon_s}
  -\pair{f}{\varepsilon_t}
  \frac{\pair{\mu_0*g}{\varepsilon_s}}{\pair{\mu_0*g}{\varepsilon_t}}
  \Bigr|.
  \end{align*}

  Now let $D=V\times (V\cdot \check U_0)$, and consider the function
  $\psi\colon D \rightarrow \R$ defined by
  \begin{equation}
    \label{eq:psi_def}
    \psi(s,t)=\pair{f}{\varepsilon_t}
    \frac{\pair{\mu_0*g}{\varepsilon_s}}{\pair{\mu_0*g}{\varepsilon_t}}.
  \end{equation}
  It is continuous and its domain $D$ is compact. Let us show that for
  a given $\epsilon>0$ there is $U_f\in \comps_e$, $U_f\subset U_0$,
  such that the variation of the function $\psi$ on $s\check U_f\times
  t \check U_f$ will be less than $\epsilon$,
  \begin{equation}
    \label{eq:var<e}
    \var_{s\check U_f\times t\check U_f} \psi =
      \sup_{(s_1,t_1), (s_2,t_2)\in s\check U_f\times t\check U_f}
    |\psi(s_1,t_1)-\psi(s_2,t_2)| < \epsilon,
  \end{equation}
   for all $(t,s)\in D$. 

  Indeed, if this were not the case, then there would exist
  $\epsilon_0>0$ such that for each $U \in \comps_e$ there would be a
  point $(s_U, t_U) \in D$ with $\var_{s_U\check U\times t_U\check U}
  \psi\geq \epsilon_0$. Consider the net $\bigl((s_U, t_U)\bigr)_{U
    \in \comps_e}$ in the compact space $D$, and let
  $\bigl((s_{\gamma}, t_{\gamma})\bigr)_{\gamma \in \Gamma}$ be a
  convergent subnet, that is, $\lim_\gamma
  (s_{\gamma},t_{\gamma})=(s^*,t^*)\in D$.

  Since $\psi$ is continuous in the point $(s^*,t^*)$ there are open
  neighborhoods $W_{s^*}$ of the point $s^*$ and $W_{t^*}$ of the
  point $t^*$ such that 
  \begin{equation}
    \label{eq:psi_variation}
    \var_{W_{s^*}\times W_{t^*}} \psi
    <\epsilon_0.
  \end{equation}
   Choose $F_{s^*}\in \comps_{s^*}$, $F_{s^*} \subset
  W_{s^*}$, and $F_{t^*} \in \comps_{t^*}$, $F_{t^*} \subset W_{t^*}$,
  and using $(H_7)$ we let $U_1 \in \comps_e$, $U_1\subset U_0$, be
  such that $F_{s^*}\cdot \check U_1 \subset W_{s^*}$ and $F_{t^*}
  \cdot \check U_1 \subset W_{t^*}$.

  Since $\lim_\gamma (s_{\gamma}, t_\gamma)=(s^*,t^*)$, there is
  $\gamma^0\in \Gamma$ such that $(s_{\gamma}, t_\gamma) \in
  F_{s^*}\times F_{t^*}$ for all $\gamma\geq \gamma^0$. Choose $\gamma
  \geq \gamma^0$ such that $U_\gamma \subset U_1$. But then, by the
  construction of $(s_\gamma, t_\gamma)$, we have that
  $\var_{s_\gamma\check U_\gamma \times t_\gamma \check U_\gamma} \psi
  \geq \epsilon_0$, which is a contradiction
  to~(\ref{eq:psi_variation}), since
  \begin{equation*}
    s_\gamma \check U_\gamma \times t_\gamma\check U_\gamma \subset
  (F_{s^*}\cdot \check U_1) \times (F_{t^*} \cdot \check U_1)
  \subset W_{s^*} \times W_{t^*}.
  \end{equation*}
  Hence~(\ref{eq:psi_variation}) holds for all $(s,t)\in D$.
  
  Setting $U_f=U_\gamma$ and using the obtained
  estimate~(\ref{eq:var<e}) and definition~(\ref{eq:psi_def}) of
  $\psi$ we get that
  \begin{equation*}
    \sup_{(s_1,t_1), (s_2,t_2)\in s\check U_f\times t\check U_f}
    \Bigl|\pair{f}{\varepsilon_{t_1}}
    \frac{\pair{\mu_0*g}{\varepsilon_{s_1}}}
    {\pair{\mu_0*g}{\varepsilon_{t_1}}}
    -
    \pair{f}{\varepsilon_{t_2}}
    \frac{\pair{\mu_0*g}{\varepsilon_{s_2}}}
    {\pair{\mu_0*g}{\varepsilon_{t_2}}}
    \Bigr|
    <\epsilon
  \end{equation*}
  for all $(s,t) \in D$.  Letting $t_1=t=s=s_1=s_2$ we have
  \begin{equation*}
    \sup_{t_2 \in s\check U_f} \Bigl|\pair{f}{\varepsilon_s}-
    \pair{f}{\varepsilon_{t_2}} 
      \frac{\pair{\mu_0 *g}{\varepsilon_s}}
      {\pair{\mu_0*g}{\varepsilon_{t_2}}}\Bigr|<\epsilon,
  \end{equation*}
  and observing that $\supp(\varepsilon_s *\check g)\subset s\check
  U_f$ we obtain the needed estimate~(\ref{eq:main_identity}).
\end{proof}

\begin{lemma}\label{L:main_estimate}
  Let $f \in \fcc_+^*$ and $\mu \in \meas_c$ be nonnegative. Then for
  any $\epsilon>0$ there exists $U\in \comps_e$ such that
  \begin{equation}
    \label{eq:main_estimate}
    (1-\epsilon) \|\mu\|\, \tilde \chi_g(f)<
    \pair{f}{\mu * \tilde \chi_g}< (1+\epsilon) \|\mu\|\, \tilde \chi_g(f),
  \end{equation}
  for arbitrary $g \in \fcc^*_+(U)$ satisfying $\check g=g$.
\end{lemma}

\begin{proof}
  Fix $U_0\in \comps_e$.  For $\epsilon_1>0$, let $U_1\in \comps_e$,
  $U_1 \subset U_0$, be such that~(\ref{eq:main_identity}) holds for
  any $g \in \fcc^*_+(U_1)$.  Since $\supp((f\chi_g)*g)=\supp(f)\cdot
  \supp(g)\subset \supp(f)\cdot U_0$, setting $K=\supp(f)\cdot U_0$
  and using~(\ref{eq:main_identity}) gives that
  \begin{equation*}
    (f\tilde \chi_g)*g-\epsilon_1 1_K
    <f<(f\tilde \chi_g)*g+\epsilon_1 1_K.
  \end{equation*}
  Hence,
  \begin{multline}
        \pair{(f\tilde \chi_g)*g}{\mu* \tilde \chi_g}-
      \epsilon_1 \pair{1_K}{\mu*\tilde \chi_g}<
      \pair{f}{\mu*\tilde \chi_g}\\
      <
      \pair{(f\tilde \chi_g)*g}{\mu* \tilde \chi_g}+
      \epsilon_1 \pair{1_K}{\mu*\tilde \chi_g}.\label{eq:2.1}
  \end{multline}

  Consider the first term in the left- and right-hand sides of~(\ref{eq:2.1}),
  \begin{equation*}
      \pair{(f\tilde \chi_g)*g}{\mu* \tilde \chi_g}=
      \pair{\check \mu * (f\tilde \chi_g)*g}{\tilde \chi_g}.
  \end{equation*}
  Denote $L_1=\supp(\check \mu * (f\tilde \chi_g)*g)=\supp(\check \mu)\cdot \supp(f)\cdot \supp(g)$, which is a
  compact set. Then
  \begin{equation}
    \label{eq:2.1++}
    \pair{\check \mu * (f\tilde \chi_g)*g}{\tilde \chi_g}=
    \pair{\check \mu * (f\tilde \chi_g)*g}{1_{L_1}\tilde \chi_g}=
    \pair{(1_{L_1}\tilde \chi_g)*\check g }{\check \mu*(f\tilde \chi_g)}.
  \end{equation}
  Let $\epsilon_2>0$ be arbitrary. Let $U_2\in \comps_e$, $U_2\subset
  U_0$, be such that
  \begin{equation*}
    \|1_{L_1}- (1_{L_1}\tilde \chi_g)*g\|<\epsilon_2
  \end{equation*}
  for any $g \in \fcc_+^*(U_2)$. If $L_2=\supp((1_{L_1} \tilde
  \chi_g)*g)= \supp(1_{L_1})\cdot \supp(g)$, then it follows from the last
  estimate that
  \begin{equation*}
    1_{L_1} - \epsilon_2 1_{L_2}<
    (1_{L_1}\tilde \chi_g)*g < 1_{L_1} + \epsilon_2 1_{L_2}.
  \end{equation*}
  Thus using that $\check g=g$ and the last inequality we get
  \begin{equation}
    \label{eq:2.1+-}
    \pair{1_{L_1} - \epsilon_2 1_{L_2}}{\check \mu*(f\tilde \chi_g)}<
    \pair{(1_{L_1}\tilde \chi_g)*\check g}{\check \mu*(f\tilde \chi_g)} 
    < \pair{1_{L_1} + \epsilon_2 1_{L_2}}{\check \mu*(f\tilde \chi_g)}.
  \end{equation}
  Since $\supp(\check \mu*(f\tilde \chi_g))=\supp(\check \mu)\cdot
  \supp(f) \subset L_1\subset L_2$, we have
  \begin{equation*}
    \pair{1_{L_1}}{\check \mu*(f\tilde \chi_g)}=
    \pair{1_Q}{\check \mu*(f\tilde \chi_g)}=
    \check \mu(1_Q) \tilde \chi_g(f)=\|\mu\|\, \tilde \chi_g(f).
  \end{equation*}
  In the same way,
  \begin{equation*}
    \pair{1_{L_2}}{\check \mu*(f\tilde \chi_g)}=
    \|\mu\| \,\tilde \chi_g(f).
  \end{equation*}
  Hence, using the obtained values in inequalities~(\ref{eq:2.1+-}) we get
  from~(\ref{eq:2.1++}) an estimate for the first term
  in~(\ref{eq:2.1}),
  \begin{equation}
    \label{eq:2.11}
    (1-\epsilon_2)\,\|\mu\|\, \tilde \chi_g(f) <
    \pair{(f \tilde \chi_g)*g}{ \mu*\tilde \chi_g}<
    (1+\epsilon_2)\,\|\mu\| \,\tilde \chi_g(f) .
  \end{equation}

  Consider now the second term in~(\ref{eq:2.1}),
  \begin{equation*}
    \epsilon_1\pair{1_K}{\mu*\tilde \chi_g}=
    \epsilon_1\pair{\check \mu*1_K}{\tilde \chi_g}.
  \end{equation*}
  If $K_1=\supp(\check \mu *1_K)$, which is compact, then
  \begin{eqnarray}
    \epsilon_1\pair{\check \mu*1_K}{\tilde \chi_g}&=&
    \epsilon_1\pair{\check \mu*1_K}{1_{K_1}\tilde \chi_g}=
    \epsilon_1\pair{1_K}{\mu*(1_{K_1} \tilde \chi_g)}
    \nonumber\\
    &<&
    \epsilon_1\pair{1}{\mu*(1_{K_1} \tilde \chi_g)}
    =\epsilon_1 \mu(1) \,\tilde \chi_g(1_{K_1})
    \nonumber\\
    &=&\epsilon_1 \|\mu\|\, \tilde \chi_g(1_{K_1}).
    \label{eq:2.3}
  \end{eqnarray}
  If we took $\epsilon_1= \frac{ \tilde \chi_g(f)}{2 \tilde
    \mu(1_{K_1})} \epsilon$, $\epsilon_2=\frac{1}{2}\epsilon $, $g\in
  \fcc_+^*(U)$, where $U=U_1\cap U_2$, using~(\ref{eq:2.11})
  and~(\ref{eq:2.3}) in~(\ref{eq:2.1}) we would arrive
  at~(\ref{eq:main_estimate}).
\end{proof}

\begin{lemma}
  \label{L:mu_U-definition+bounds}
  Let $f_0 \in \fcc^*_+$ be fixed. For $U\in \comps_e$ and $g\in
  \fcc_+^*(U)$ satisfying $\check g=g$, define
  \begin{equation}
    \label{eq:mu-definition}
    \chi_g=\frac{\tilde \chi_g}{\tilde \chi_g(f_0)},
  \end{equation}
  where $\tilde \chi_g$ is defined in~(\ref{eq:mu_U_definition}). Then
    for any $f \in \fcc_+^*$ there exist positive $a_f, b_f \in \R$ and 
   $U_f\in \comps_e$ such that
  \begin{equation}
    \label{eq:m_u-bounds}
    a_f<\chi_g(f)<b_f
  \end{equation}
  for any $g\in \fcc_+^*(U_f)$, $\check g=g$.
\end{lemma}

\begin{proof}
  For functions $f_0,f \in \fcc_+$, it follows
  from~\cite[III~(p.~159)]{spector78:_meas_inv} that there exist
  positive measures $\mu_1, \mu_2 \in \meas_c$ such that
  \begin{equation*}
    f<\mu_1*f_0\qquad\mbox{and}\qquad f_0<\mu_2*f.
  \end{equation*}
  Then, fixing $\epsilon=1$ and letting $U_1\in \comps_e$ be such
  that~(\ref{eq:main_estimate}) holds for $f_0$, $\mu_1$, and
  arbitrary $g \in \fcc_+^*(U_1)$, we find that
  \begin{equation*}
    \tilde \chi_g(f) <\pair{\mu_1*f_0}{\tilde  \chi_g}=
    \pair{f_0}{\check \mu_1*\tilde \chi_g}
    <2 \|\mu_1\|\, \tilde \chi_g(f_0),
  \end{equation*}
  or
  \begin{equation*}
    \chi_g(f)=\frac{\tilde \chi_g(f)}{\tilde \chi_g(f_0)}<2\|\mu_1\|.
  \end{equation*}
  If now $U_2\in \comps_e$ is such that~(\ref{eq:main_estimate}) holds
  for $f$, $\mu_2$, and arbitrary $g\in \fcc_+^*(U_2)$, then we
  similarly get that
  \begin{equation*}
    \tilde \chi_g(f_0)<2\|\mu_2\| \, \tilde \chi_g(f),
  \end{equation*}
  or
  \begin{equation*}
    \chi_g(f)=\frac{\tilde \chi_g(f)}{\tilde \chi_g(f_0)}> \frac{1}{2\|\mu_2\|}.
  \end{equation*}
  Hence, we can take $a_f=\frac{1}{2\|\mu_2\|}$, $b_f=2\|\mu_1\|$, 
  and $U_f=U_1\cap U_2$.
\end{proof}

\begin{lemma}
  \label{L:existence_of_limit}
  For each $U \in \comps_e$, choose $g \in \fcc_+^*(U)$ satisfying
  $\check g=g$.  Then the net $(\chi_{g})_{U \in \comps_e}$ is
  $\sigma(\meas,\fcc)$-convergent, that is, the limit $\lim_{U}
  \chi_{g}(f)$ exists for any $f \in \fcc$.
\end{lemma}

\begin{proof}
  It is sufficient to prove that the net $(\chi_{g}(f))_{U\in
    \comps_e}$ is Cauchy for $f\in \fcc_+^*$. 

  First of all, it follows from~(\ref{eq:mu-definition}) that
  $\chi_{k g}=\chi_{g}$ for any real $k>0$. So we can
  assume that all considered $g$ are normalized so that $\tilde
  \chi_{g}(f_0)=1$, hence $\tilde \chi_{g}=\chi_{g}$.

  Now, let $V\in \comps_e$ be fixed, and let $\epsilon$,
  $0<\epsilon<1$, be arbitrary.

  Let $U_0\in \comps_e$, $U_0\subset V$, be such
  that~(\ref{eq:main_identity}) holds for the functions $f, f_0$ and
  the measure $ \chi_{g_{0}}$, as well as $\chi_{g_{0}}(f)$ satisfies
  estimate~(\ref{eq:m_u-bounds}), for any $g_0 \in \fcc_+^*(U_0)$. Then
  \begin{gather}
    (f \chi_{g_{0}}) * g_{0} - \epsilon 1_K<
    f
    <
    (f \chi_{g_{0}}) * g_{U_0} + \epsilon 1_K,
    \label{eq:<>f} \\
        (f_0 \chi_{g_{0}}) * g_{0} - \epsilon 1_K<
    f_0
    <
    (f_0 \chi_{g_{0}}) * g_{0} + \epsilon 1_K,
    \label{eq:<>f0}
  \end{gather}
  where we set $K=\bigl(\supp(f)\cup \supp(f_0)\bigr)\cdot V$.
  Considering~(\ref{eq:<>f}) we get that
  \begin{equation}
    \label{eq:mu_1(f)<>}
    \pair{(f \chi_{g_{0}}) * g_{0}}{ \chi_{g}} - 
    \epsilon  \chi_{g}(1_K)<
    \chi_{g}(f) 
    <
    \pair{(f \chi_{g_{0}}) * g_{0}}{ \chi_{g}} 
    + \epsilon  \chi_{g}(1_K)
  \end{equation}
  for any $U \in \comps_e$ and $g\in \fcc_+^*(U)$.  Now, consider
  \begin{equation*}
    \pair{(f \chi_{g_{0}})*g_{0}}{ \chi_{g}}=
    \pair{g_{0}}{(f \chi_{g_{0}})\, \check{} \, *
       \chi_{g}},
  \end{equation*}
  and let $U\in \comps_e$ be such that the measure $ \chi_{g}$, $g \in
  \fcc_+^*(U)$, would satisfy~(\ref{eq:main_estimate}) for the
  function $g_{0}$ and the measure $(f \chi_{g_{0}})\,\check {}\,
  $. Then we have
  \begin{equation*}
    (1-\epsilon)  \chi_{g_{0}}(f)  \chi_{g}(g_{0})<
    \pair{(f \chi_{g_{0}})*g_{0}}{ \chi_{g}}<
    (1+\epsilon)  \chi_{g_{0}}(f)  \chi_{g}(g_{0}),
  \end{equation*}
  which, together with~(\ref{eq:mu_1(f)<>}), gives
  \begin{equation}
    (1-\epsilon)  \chi_{g_{0}}(f)  \chi_{g}(g_{0})
    -\epsilon  \chi_{g}(1_K) <  \chi_{g}(f)<
    (1+\epsilon)  \chi_{g_{0}}(f)  \chi_{g}(g_{0})
    +\epsilon  \chi_{g}(1_K).\label{eq:mu_1(f)<>!}
  \end{equation}
  Assuming that $U$ was chosen such that the measure $ \chi_{g}$
  satisfies~(\ref{eq:main_estimate}) also for the function $g_{0}$ and
  the measure $(f_0 \chi_{g_{0}})\, \check{}\,$, we get similar
  estimates for $ \chi_{g}(f_0)$,
  \begin{equation}
    (1-\epsilon)  \chi_{g_{0}}(f_0)  \chi_{g}(g_{0})
    -\epsilon  \chi_{g}(1_K) <  \chi_{g}(f_0)<
    (1+\epsilon)  \chi_{g_{0}}(f_0)  \chi_{g}(g_{0})
    +\epsilon  \chi_{g}(1_K).\label{eq:mu_1(f_0)<>!}
  \end{equation}
  Since $g$ is normalized so that $ \chi_{g}(f_0)=1$, these
  inequalities yield the estimate
  \begin{equation*}
    \frac{1-\epsilon  \chi_{g}(1_K)}{1+\epsilon}<
     \chi_{g}(g_{0}) <\frac{1+\epsilon  \chi_{g}(1_K)}
    {1-\epsilon}.
  \end{equation*}
  If we apply this estimate to~(\ref{eq:mu_1(f)<>!}) we get
  \begin{multline*}
    \label{eq:<>final}
    \frac{1-\epsilon}{1+\epsilon} \bigl(1-\epsilon
    \chi_{g}(1_K)\bigr)\chi_{g_{0}}(f) -\epsilon \chi_{g}(1_K)<
    \chi_{g}(f)\\
    <\frac{1+\epsilon}{1-\epsilon} \bigl(1+\epsilon
    \chi_{g}(1_K)\bigr)\chi_{g_{0}}(f) +\epsilon \chi_{g}(1_K).
  \end{multline*}
  Finally, assuming that $U$ was chosen so that $\chi_{g}(1_K)$
  satisfies estimate~(\ref{eq:m_u-bounds}), we get
  \begin{equation}
    \label{eq:final_estimate_U}
    \frac{1-\epsilon}{1+\epsilon} (1-\epsilon b_{1_K}) \chi_{g_{0}}(f)-
    \epsilon b_{1_K}<
    \chi_{g}(f)
    <
    \frac{1+\epsilon}{1-\epsilon} (1+\epsilon b_{1_K}) \chi_{g_{0}}(f)+
    \epsilon b_{1_K}.
  \end{equation}
  This implies that for any $U_1 \subset U$, $U_2\subset U$, and
  $g_1\in\fcc_+^*(U_1)$, $g_2\in \fcc_+^*(U_2)$,
  \begin{equation}
    \label{eq:||estimate}
    |\chi_{g_{1}} (f) - \chi_{g_{2}}(f)|<\frac{2\epsilon}{1-\epsilon^2}
    \bigl((1+ (1+\epsilon^2) b_{1_K}\bigr) \chi_{g_{0}}(f) + 
    2\epsilon b_{1_K}.
  \end{equation}
  Since $U_0$ was chosen so that
  $\chi_{g_{0}}(f)<b_{f}$, it follows from~(\ref{eq:||estimate})  that
  $\bigl(\chi_{g}(f)\bigr)_{U\in \comps_e}$ is a Cauchy net.
\end{proof}

\begin{theorem*}
  On any \textup{(}locally compact\textup{)} hypergroup $Q$ there
  exists a left invariant measure $\chi$ with $\supp(\chi)=Q$.
\end{theorem*}

\begin{proof}
  For any $f \in \fcc_+^*$, set $\chi(f)=\lim_U \chi_{g}(f)$ as in
  Lemma~\ref{L:existence_of_limit}, and extend it by linearity to a
  linear functional on $\fcc$.

  Let $K$ be a compact subset of $Q$ and $f_1, f_2 \in \fcc(K)$.
  Then, for any $U\in \comps_e$ and $g \in \fcc_+^*(U)$, we have
  \begin{equation*}
    |\chi_{g}(f_1)-\chi_{g}(f_2)|=|\chi_{g}(f_1-f_2)|=
    |\chi_{g}(1_K(f_1-f_2))|\leq \chi_{g}(1_K) \|f_1-f_2\|.
  \end{equation*}
  This shows that, eventually,
  \begin{equation*}
    |\chi_{g}(f_1)-\chi_{g}(f_2)|\leq b_{1_K} \|f_1-f_2\|,
  \end{equation*}
  hence $\chi$ is a bounded linear functional on $\fcc(K)$ for
  arbitrary compact $K$, and thus a Radon measure.

  By setting $\mu=\check \varepsilon_s$ in
  inequality~(\ref{eq:main_estimate}) for $\tilde \chi_{g}$, dividing
  it by $\tilde \chi_{g}(f_0)$ and using that $\|\check
  \varepsilon_s\|=1$, we get
  \begin{equation*}
    (1-\epsilon)\chi_{g}(f)<\pair{f}
    {\check \varepsilon_s*\chi_{g}}=
    \pair{\varepsilon_s*f}{\chi_{g}}<    
    (1+\epsilon)\chi_{g}(f).
  \end{equation*}
  Since this eventually holds for any $\epsilon>0$, we see that
  $\chi$ satisfies~(\ref{eq:left-invariant-measure-def}), hence it
  is a left invariant measure. It is clear that $\supp(\chi)=Q$,
  since, for any $f \in \fcc_+^*$, we eventually have
  $\chi_{g}(f)>a_f>0$ by Lemma~\ref{L:mu_U-definition+bounds}.
\end{proof}

\bigskip \textit{Acknowledgments.} The author is very grateful to
A.~A.~Kalyuzhnyi and G.~B.~Podkolzin for numerous and fruitful
discussions of the subject of the article.

\end{document}